\documentclass[11pt]{amsart}

\usepackage{amsmath,amsthm,amsfonts,amssymb,mathrsfs}
\usepackage[margin=1in]{geometry}
\usepackage[all,cmtip]{xy}
\usepackage{hyperref,cite,enumitem}

\newcommand{\set}[1]{\left\{ #1 \right\}}
\newcommand{\abs}[1]{\left| #1 \right|}
\newcommand{\wt}[1]{\widetilde{ #1}}
\newcommand{\ol}[1]{\overline{#1}}

\newcommand{\mbar}{\ol{m}}
\newcommand{\wtm}{\wt{m}}
\newcommand{\ybar}{\ol{y}}
\newcommand{\wtx}{\wt{x}}
\newcommand{\wty}{\wt{y}}

\DeclareMathOperator{\cls}{cls}
\DeclareMathOperator{\Ext}{Ext}

\DeclareMathOperator{\opH}{H}
\newcommand{\Hbul}{\opH^\bullet}
\DeclareMathOperator{\Hom}{Hom}
\DeclareMathOperator{\id}{id}
\DeclareMathOperator{\im}{im}
\DeclareMathOperator{\Lie}{Lie}

\DeclareMathOperator{\Tor}{Tor}
\newcommand{\ve}{\varepsilon}

\newcommand{\C}{\mathbb{C}}
\newcommand{\F}{\mathbb{F}}
\newcommand{\N}{\mathbb{N}}
\newcommand{\Q}{\mathbb{Q}}
\newcommand{\Z}{\mathbb{Z}}

\newcommand{\Fp}{\F_p}

\newcommand{\ca}{\mathcal{A}}
\newcommand{\cb}{\mathcal{B}}
\newcommand{\cz}{\mathcal{Z}}

\newcommand{\g}{\mathfrak{g}}

\newcommand{\gfp}{\g_{\Fp}}

\newcommand{\A}{\mathsf{A}}
\newcommand{\B}{\mathsf{B}}

\newcommand{\oa}{\otimes_\A}

\newcommand{\oca}{\otimes_\ca}
\newcommand{\ocb}{\otimes_\cb}
\newcommand{\ocz}{\otimes_\cz}

\newcommand{\Ce}{\C_\ve}
\newcommand{\Qe}{\Q(\ve)}
\newcommand{\Cg}{\C_\g}
\newcommand{\Cq}{\C(q)}

\newcommand{\Ua}{U_\A}

\newcommand{\Uca}{U_\ca}
\newcommand{\Ucb}{U_\cb}
\newcommand{\Ucz}{U_\cz}
\newcommand{\Ue}{U_\ve}
\newcommand{\Ug}{U(\g)}
\newcommand{\Uq}{U_q}
\newcommand{\Uqg}{\Uq(\g)}

\newtheorem{theorem}{Theorem}[section]
\newtheorem{lemma}[theorem]{Lemma}
\newtheorem{corollary}[theorem]{Corollary}
\newtheorem{proposition}[theorem]{Proposition}
\newtheorem{conjecture}[theorem]{Conjecture}

\theoremstyle{definition}

\theoremstyle{remark}

\numberwithin{equation}{section}

\title{Cohomology rings for Quantized Enveloping Algebras}

\author{Christopher M.\ Drupieski}
\address{Department of Mathematics\\ University of Georgia \\ Athens\\ GA~30602-7403, USA}
\thanks{The author was supported in part by NSF VIGRE grant DMS-0738586}
\email{cdrup@math.uga.edu}


\begin{document}

\begin{abstract}
We compute the structure of the cohomology ring for the quantized enveloping algebra (quantum group) $U_q$ associated to a finite-dimensional simple complex Lie algebra $\g$. We show that the cohomology ring is generated as an exterior algebra by homogeneous elements in the same odd degrees as generate the cohomology ring for the Lie algebra $\g$. Partial results are also obtained for the cohomology rings of the non-restricted quantum groups obtained from $U_q$ by specializing the parameter $q$ to a non-zero value $\ve \in \C$.
\end{abstract}

\maketitle

\section{Introduction}

\subsection{}

Let $G$ be a simple compact connected Lie group of dimension $d$. It is a famous theorem from algebraic topology that the homology and cohomology algebras for $G$ are exterior algebras over graded subspaces concentrated in odd degrees \cite{Samelson:1952}. By a result of Cartan, the homology and cohomology algebras for $G$ identify with those for its Lie algebra $\g$, so we get also the ring structure of the Lie algebra cohomology ring $\Hbul(\g,\C) = \Hbul(\Ug,\C)$. Here $\Ug$ denotes the universal enveloping algebra of $\g = \Lie(G)$. In recent years there has been much interest in homological and cohomological properties for various classes of noetherian Hopf algebras \cite{Brown:1997,Brown:2008,Chemla:2004}, important examples of which being the universal enveloping algebras and quantized enveloping algebras associated to a finite-dimensional simple complex Lie algebra. A common theme to some of the recent work has been the desire to generalize Poincar\'{e} duality to these classes of noetherian Hopf algebras \cite{Brown:2008,Kowalzig:2010}.

Let $q$ be an indeterminate, and set $k = \Cq$. Let $\Uq$ be the quantized enveloping algebra over $k$ associated to the finite-dimensional simple complex Lie algebra $\g$. Though the above cited works provide general results relating the dimensions of the homology and cohomology groups
\[
\opH_n(\Uq,k) = \Tor_n^{\Uq}(k,k) \qquad \text{and} \qquad \opH^n(\Uq,k) = \Ext_{\Uq}^n(k,k),
\]
namely, $\dim_k \opH^n(\Uq,k) = \dim_k \opH_{d-n}(\Uq,k)$, there have been no explicit calculations of the dimensions of these groups, nor of the ring structure for the cohomology ring $\Hbul(\Uq,k)$. Similarly, one would like to know the dimension and ring structure of the cohomology ring $\Hbul(\Ue,\C)$ associated to the quantized enveloping algebra $\Ue$ with parameter $q$ specialized to a value $\ve \in \C^\times := \C - \set{0}$.

In this paper we show that the cohomology ring $\Hbul(\Uq,k)$ is an exterior algebra generated by homogeneous elements in the same odd degrees as for $\Hbul(\Ug,\C)$, and thus for each $n \in \N$ that the cohomology group $\opH^n(\Uq,k)$ for $\Uq$ is of the same dimension as the corresponding group for $\Ug$. Our proof relies on an integral form $\Ua$ for $\Uq$, which enables us to relate, via the universal coefficient theorem, cohomology for $\Uq$ to that for $\Ug$. The main steps of this argument are carried out in Sections \ref{subsection:Uafree} and \ref{subsection:ringstructure}. A key step in the proof is the calculation of the restriction maps in Lie algebra cohomology associated to an inclusion $F \subset E$ of simple Lie algebras; see Section \ref{subsection:restriction}. Finally, assuming $\ve \in \C$ is a root of unity of sufficiently large prime order $p$, we obtain the structure of the cohomology ring $\Hbul(\Ue,\C)$ for the quantized enveloping algebra $\Ue$. This last computation exploits a connection between $\Ue$ and the characteristic $p$ universal enveloping algebra of $\g$.

\subsection{Notation} \label{subsection:notation}

Let $\N = \set{0,1,2,3,\ldots}$ denote the set of non-negative integers. Let $(a_{ij})$ be the $r \times r$ Cartan matrix associated to the finite-dimensional simple complex Lie algebra $\g$, and let $(d_1,\ldots,d_r) \in \N^r$ be the unique vector such that $\gcd(d_1,\ldots,d_r) =1$ and the matrix $(d_ia_{ij})$ is symmetric.  The ordering of the rows and columns for the matrix $(a_{ij})$ corresponds to a labeling of the Dynkin diagram associated to $\g$; we assume this is done as in Bourbaki \cite[Plates I--IX]{Bourbaki:2002}.

Let $\C[q,q^{-1}]$ be the Laurent polynomial ring over $\C$ in the indeterminate $q$, and let $k = \Cq$ be its quotient field. Then the quantized enveloping algebra (or quantum group) associated to $\g$ is the $k$-algebra $\Uq = \Uqg$ defined by the generators $\{ E_i,F_i,K_i^{\pm 1}: i \in [1,r] \}$ and the relations given in \cite[(1.2.1--1.2.5)]{De-Concini:1990}. The algebra $\Uq$ is also a Hopf algebra via the maps in \cite[(1.2.6--1.2.8)]{De-Concini:1990}.

For $n \in \Z$ and $d \in \N$, put $[n]_d = (q^{nd}-q^{-nd})/(q^d-q^{-d}) \in \Z[q,q^{-1}]$. Given $\ell \in \N$, let $\phi_\ell \in \Z[q]$ be the $\ell$-th cyclotomic polynomial. Now define $S \subset \Z[q,q^{-1}]$ to be the multiplicatively closed set generated by
\begin{align}
&\set{1} & \text{if $\g$ has Lie type $ADE$}, \notag \\
&\set{\phi_4,\phi_8} & \text{if $\g$ has Lie type $BCF$}, \label{eq:badcyclotomics} \\
&\set{\phi_3,\phi_4,\phi_6,\phi_9,\phi_{12},\phi_{18}} & \text{if $\g$ has Lie type $G_2$}. \notag
\end{align}
Then the generators for $S$ are precisely the irreducible factors of $[n]_{d_i}$ in $\Z[q,q^{-1}]$ when $1 \leq n \leq \abs{a_{ij}}$ and $i \neq j$. Set $\cz = S^{-1} \Z[q,q^{-1}]$ and $\ca = S^{-1}\C[q,q^{-1}]$, the localizations of $\Z[q,q^{-1}]$ and $\C[q,q^{-1}]$ at $S$.

Let $\Ucz$ be the $\cz$-subalgebra of $\Uq$ generated by the set $\{E_i,F_i,K_i^{\pm 1}: i \in [1,r]\}$. This algebra is the De Concini--Kac integral form of $\Uq$ (over $\cz$). Given a $\cz$-algebra $B$, set $U_B = \Ucz \ocz B$. In particular, $\Uca$ is the $\ca$-subalgebra of $\Uq$ generated by $\{E_i,F_i,K_i^{\pm 1}: i \in [1,r]\}$. Given $\ve \in \C^\times$ with $f(\ve) \neq 0$ for all $f \in S$, write $\Ce$ for the field $\C$ considered as an $\ca$-algebra via the map $q \mapsto \ve$, and set $\Ue = \Uca \oca \Ce \cong \Uca/(q-\ve)\Uca$. Equivalently, $\Ue \cong \Ucz \otimes_\cz \Ce$. If $\ve^{2d_i} \neq 1$ for all $i \in [1,r]$, then $\Ue$ is the $\C$-algebra with the same generators and relations as $\Uq$, but with $q$ replaced by $\ve$. For this reason, we call $\Ue$ a specialization of $\Uq$.

For each $i \in [1,r]$, let $T_i$ be the braid group operator on $\Uq$ as defined in \cite[\S 1.6]{De-Concini:1990}, and let $\Phi$ be the root system associated to $\g$. Then for each positive root $\beta \in \Phi^+$, there exist root vectors $E_\beta,F_\beta \in \Uq$, defined in terms of the $T_i$ \cite[\S 1.7]{De-Concini:1990}. By the definition of the denominator set $S$, the $T_i$ restrict to automorphisms of the algebra $\Ucz$, so also $E_\beta,F_\beta \in \Ucz \subset \Uca$. (This is why we work with the coefficient rings $\cz$ and $\ca$ instead of with $\Z[q,q^{-1}]$ and $\C[q,q^{-1}]$.)

\section{Lie algebra cohomology}

\subsection{An isomorphism with \texorpdfstring{$\mathbf{U_1}$}{U1} cohomology}

We begin with an observation on the relationship between the cohomology spaces for the universal enveloping algebra $\Ug$ and for the specialization $U_1$. Recall from \cite[Proposition 1.5]{De-Concini:1990} that $U_1$ is a central extension of $\Ug$ by the group algebra over $\C$ for the finite group $(\Z/2\Z)^r$, so there exists a surjective Hopf algebra homomorphism $U_1 \rightarrow \Ug$.

\begin{lemma} \label{lemma:U1toUgiso}
The homomorphism $U_1 \rightarrow \Ug$ induces an algebra isomorphism
\[
\Hbul(\Ug,\C) \stackrel{\sim}{\rightarrow} \Hbul(U_1,\C).
\]
\end{lemma}

\begin{proof}
Set $G=(\Z/2\Z)^r$, and consider the Lyndon--Hochschild--Serre (LHS) spectral sequence for the algebra $U_1$ and its normal Hopf-subalgebra isomorphic to $\C G$:
\[
E_2^{i,j} = \opH^i(\Ug,\opH^j(\C G,\C)) \Rightarrow \opH^{i+j}(U_1,\C).
\]
The group algebra $\C G$ is a semisimple algebra, so $E_2^{i,j} = 0$ for all $j > 0$. Since the spectral sequence respects cup products, it follows that the edge map $E_2^{\bullet,0} = \Hbul(\Ug,\C) \rightarrow \Hbul(U_1,\C)$ is an algebra isomorphism.
\end{proof}

\subsection{The structure of Lie algebra cohomology}

Lemma \ref{lemma:U1toUgiso} reduces the problem of studying the cohomology ring $\Hbul(U_1,\C)$ to the classical problem of studying the cohomology ring $\Hbul(\Ug,\C)$. We summarize some details on the computation of $\Hbul(E,\C) = \Hbul(U(E),\C)$ for $E$ an arbitrary finite-dimensional reductive Lie algebra over $\C$. Our main reference is \cite[Chapters V--VI]{Greub:1976}.

Let $\Lambda^\bullet(E^*)$ denote the exterior algebra on the dual space $E^* = \Hom_\C(E,\C)$, considered as a graded complex with $E^*$ concentrated in degree $1$. The map $\Lambda^2(E) \rightarrow E$ defined by $x \wedge y \mapsto [x,y]$ induces a map $d:E^* \rightarrow \Lambda^2(E^*)$, which extends by derivations to a differential on $\Lambda^\bullet(E^*)$, also denoted $d$. Then $\Hbul(E,\C)$ is the cohomology of the complex $\Lambda^\bullet(E^*)$ with respect to the differential $d$. The space $\Lambda^\bullet(E^*)$ is also naturally an $E$-module, with $E$-action induced by the coadjoint action of $E$ on $E^*$. Let $\Lambda^\bullet(E^*)^E$ denote the space of $E$-invariants in $\Lambda^\bullet(E^*)$. Then the inclusion $\Lambda^\bullet(E^*)^E \hookrightarrow \Lambda^\bullet(E^*)$ induces an algebra isomorphism $\Lambda^\bullet(E^*)^E \cong \Hbul(E,\C)$.

The addition map $E \times E \rightarrow E$, $(x,y) \mapsto x+y$, induces on $\Lambda^\bullet(E^*)$ the structure of a bialgebra, and the bialgebra structure restricts to one on $\Lambda^\bullet(E^*)^E \cong \Hbul(E,\C)$. Then $\Hbul(E,\C)$ is generated as an algebra by its subspace of primitive elements, which we denote by $P_E$. The subspace $P_E$ is concentrated in odd degrees, and the induced map $\Lambda(P_E) \rightarrow \Lambda^\bullet(E^*)^E \rightarrow \Hbul(E,\C)$ is an algebra isomorphism.

\begin{theorem} \label{theorem:exterioralgebra} \textup{\cite{Greub:1976,Samelson:1952}}
Let $\g$ be a finite-dimensional simple complex Lie algebra. Then $\Hbul(\Ug,\C)$ is an exterior algebra generated by homogeneous elements in the odd degrees listed in Table \ref{table:degrees}.
\end{theorem}

\begin{table}[htbp]
\begin{tabular}{ll}
\hline
Type & Degrees \\
\hline
$A_r$ & $3,5,7,\dots,2r+1$ \\
$B_r$ & $3,7,11,\dots,4r-1$ \\
$C_r$ & $3,7,11,\dots,4r-1$ \\
$D_r$ ($r \geq 4$) & $3,7,11,\dots,4r-5,2r-1$ \\
$E_6$ & $3,9,11,15,17,23$ \\
$E_7$ & $3,11,15,19,23,27,35$ \\
$E_8$ & $3,15,23,27,35,39,47,59$ \\
$F_4$ & $3,11,15,23$ \\
$G_2$ & $3,11$ \\
\end{tabular}
\caption{Degrees of homogeneous generators for $\Hbul(\Ug,\C)$.} \label{table:degrees}
\end{table}

\subsection{Restriction maps} \label{subsection:restriction}

Let $E$ and $F$ be finite-dimensional reductive Lie algebras over $\C$ with $F \subseteq E$. Write $j: F \rightarrow E$ for the inclusion map. Let $W(E)$ and $W(F)$ be the Weyl groups associated to $E$ and $F$, respectively. The cohomological restriction map $\Hbul(E,\C) \rightarrow \Hbul(F,\C)$ is completely determined by the induced map $j^*:P_E \rightarrow P_F$ on the spaces of primitive elements.

Let $H \subset F$ be a Cartan subalgebra of $F$, and let $H' \subset E$ be a Cartan subalgebra of $E$ containing $H$. Let $S(E^*)$ be the ring of polynomial functions on $E$, but with the subspace $E^*$ concentrated in degree $2$. Similarly, define $S(F^*)$, $S(H^*)$, and $S(H'^*)$ to be the evenly graded rings of polynomial functions on $F$, $H$, and $H'$. The coadjoint action of $E$ on $E^*$ extends to an action of $E$ on $S(E^*)$, and similarly for $F$ on $S(F^*)$. Then the restriction map $S(E^*) \rightarrow S(F^*)$ induces a map $S(E^*)^E \rightarrow S(F^*)^F$. By \cite[\S 11.9]{Greub:1976}, the restriction maps $S(E^*) \rightarrow S(H'^*)$ and $S(F^*) \rightarrow S(H^*)$ induce isomorphisms $S(E^*)^E \cong S(H'^*)^{W(E)}$ and $S(F^*)^F \cong S(H^*)^{W(F)}$. Since $W(E)$ and $W(F)$ are finite reflection groups, the rings $S(H'^*)^{W(E)}$ and $S(H^*)^{W(F)}$ are generated by algebraically independent homogeneous elements.

By \cite[\S 6.7]{Greub:1976}, there exists a canonical linear map $\rho_E: S(E^*)^E \rightarrow \Lambda^\bullet(E^*)^E$, homogeneous of degree $-1$, and natural with respect to the inclusion $F \subseteq E$. By \cite[\S 6.14]{Greub:1976}, $\im \rho_E = P_E$, and $\ker \rho_E = (S(E^*)^E)^2$. Then
\[
j^*(P_E) \cong j^*(S(E^*)^E)/(S(F^*)^F)^2 \cong j^*(S(H'^*)^{W(E)})/(S(H^*)^{W(F)})^2,
\]
so to compute the map $j^*: P_E \rightarrow P_F$, and hence the map $j^*: \Hbul(E,\C) \rightarrow \Hbul(F,\C)$, it suffices to determine which polynomial generators for $S(H'^*)^{W(E)}$ restrict to a sum of decomposable elements in $S(H^*)^{W(F)}$.

In the following theorem we explicitly describe the cohomological restriction map $\Hbul(E,\C) \rightarrow \Hbul(F,\C)$ for certain simple pairs $(E,F)$. Specifically, let $E$ be a simple complex Lie algebra with associated root system $\Phi$, and let $\Delta = \set{\alpha_1,\ldots,\alpha_r}$ be a set of simple roots for $\Phi$, ordered as in \cite[Plates I-IX]{Bourbaki:2002}. Then we will assume that $F$ is a simple subalgebra of $E$ of rank $r-1$ corresponding to removing some simple root $\alpha_F$ from $\Delta$. For ease in stating the theorem, we identify the Lie algebras $E$ and $F$ with their respective Lie types.

\begin{theorem} \label{theorem:restrictionmap}
Let $E$, $F$, $\alpha_F$ be as in the previous paragraph. Write $\Hbul(E,\C) = \Lambda(x_{i_1},\ldots,x_{i_r})$ and $\Hbul(F,\C) = \Lambda(y_{j_1},\ldots,y_{j_{r-1}})$ as in Theorem \ref{theorem:exterioralgebra}, with the $x_i$ and $y_j$ homogeneous of degrees $i$ and $j$, respectively. If one of $E$ or $F$ is of type $D$, write $\wtx_i$ or $\wty_j$ for the generator of the last degree listed in Table \ref{table:degrees}. Then the homogeneous generators can be chosen so that the cohomological restriction map $\Hbul(E,\C) \rightarrow \Hbul(F,\C)$ admits the following description:
\begin{enumerate}
\item If $(E,F) = (A_r,A_{r-1})$ and $\alpha_F = \alpha_1$, then
\[
x_3 \mapsto y_3,\  x_5 \mapsto y_5,\  \ldots,\  x_{2r-1} \mapsto y_{2r-1},\  x_{2r+1} \mapsto 0.
\]

\item If $(E,F) = (B_r,B_{r-1})$ or $(C_r,C_{r-1})$ and $\alpha_F = \alpha_1$, then
\[
x_3 \mapsto y_3,\  x_7 \mapsto y_7,\  \ldots,\  x_{4r-5} \mapsto y_{4r-5},\  x_{4r-1} \mapsto 0.
\]

\item If $(E,F) = (D_r,D_{r-1})$, $r \geq 5$, and $\alpha_F = \alpha_1$, then
\[
x_3 \mapsto y_3,\  x_7 \mapsto y_7,\  \ldots,\  x_{4r-9} \mapsto y_{4r-9},\  x_{4r-5} \mapsto 0,\ \wtx_{2r-1} \mapsto 0.
\]

\item If $(E,F) = (D_r,A_{r-1})$ and $\alpha_F = \alpha_r$, then $\wtx_{2r-1} \mapsto y_{2r-1}$, and
\[
x_i \mapsto \begin{cases} 2y_i & \text{if $i=4j-1$ for some $j \geq 1$ with $2j \leq r$,} \\ 0 & \text{otherwise.} \end{cases}
\]

\item If $(E,F) = (E_6,D_5)$ and $\alpha_F = \alpha_6$, then
\[
x_3 \mapsto y_3,\ x_9 \mapsto \wty_9,\ x_{11} \mapsto y_{11},\ x_{15} \mapsto y_{15},\ x_{17} \mapsto 0,\ x_{23} \mapsto 0.
\]

\item If $(E,F) = (E_7,E_6)$ and $\alpha_F = \alpha_7$, then
\[
x_3 \mapsto 2y_3,\ x_{11} \mapsto 2y_{11},\ x_{15} \mapsto 2y_{15},\ x_{19} \mapsto 0,\ x_{23} \mapsto 2y_{23},\ x_{27} \mapsto 0,\ x_{35} \mapsto 0.
\]
\end{enumerate}
\end{theorem}

\begin{proof}
The various cases of the theorem are established through direct computation of either $P_E \rightarrow P_F$, $S(E^*)^E \rightarrow S(F^*)^F$, or $S(H'^*)^{W(E)} \rightarrow S(H^*)^{W(F)}$. The case $(E_6,D_5)$ is computed in \cite[(5.6)]{Toda:1974}, the case $(E_7,E_6)$ is computed in \cite[(2.3)]{Watanabe:1975}, and the remaining cases (and many others) are computed in \cite[Chapter XI.4]{Greub:1976}.
\end{proof}

\section{Cohomology for the integral form \texorpdfstring{$\Ua$}{UA}}

Next we study the cohomological properties of a certain integral form $\Ua$ of $\Uq$, to be defined in Section \ref{subsection:integralform}, which will enable us to relate cohomology for $\Uq$ to that for the Lie algebra $\g$. First we collect some results on the algebra $\Ucz$.

\subsection{A resolution of the trivial module} \label{subsection:resolution}

We begin with the following lemma, which is well-known for the Lusztig integral form of $\Uq$, though we could find no analogous statement in the literature for the De Concini--Kac integral form $\Ucz$ as we have defined it here. We thus record the result now.

\begin{lemma} \label{lemma:free}
The algebra $\Ucz$ is a free $\cz$-module. Consequently, for any $\cz$-algebra $B$, the algebra $U_B = \Ucz \ocz B$ is free over $B$.
\end{lemma}

\begin{proof}[Proof sketch]
The algebra $\Ucz$ inherits from $\Uq$ the triangular decomposition $\Ucz \cong \Ucz^+ \ocz \Ucz^0 \ocz \Ucz^-$, where $\Ucz^+$ (resp.\ $\Ucz^-$) is the $\cz$-subalgebra of $\Ucz$ generated by the $E_i$ (resp.\ $F_i$) for $i \in [1,r]$. Since $\Ucz^+$ contains for each $\beta \in \Phi^+$ the root vector $E_\beta$, it follows that $\Ucz^+$ is spanned over $\cz$ by the collection of PBW-monomials $\prod_{\beta \in \Phi^+} E_\beta^{n_\beta}$, $n_\beta \in \N$, and that these monomials form a $\cz$-basis for $\Ucz^+$; cf.\ \cite[\S1.7]{De-Concini:1990}. By symmetry, $\Ucz^-$ is also free over $\cz$. For $i \in [1,r]$, let $A_i$ be the $\cz$-subalgebra of $\Ucz^0$ generated by $\{K_i^{\pm 1}, [K_i;0] \}$. Then $\Ucz^0 \cong A_1 \ocz \cdots \ocz A_r$, so to prove the first claim it suffices to show that each $A_i$ is $\cz$-free. Observe that $K_i^{-1} = K_i - (q_i-q_i^{-1})[K_i;0]$, where $q_i = q^{d_i}$, so $A_i$ is generated as a $\cz$-algebra by $K_i$ and $[K_i;0]$. The identity also shows that $K_i^2 = 1 + (q_i - q_i^{-1})K_i[K_i;0]$, so it follows that $A_i$ is spanned over $\cz$ by the collection of elements $\set{[K_i;0]^n, K_i [K_i;0]^m: n,m \in \N}$. Now one can apply \cite[(1.5.4)]{De-Concini:1990} and \cite[6.4(b2) and 6.7(i)(c)]{Lusztig:1990a} to deduce that this set is linearly independent over $\cz$, and hence forms a $\cz$-basis for $A_i$. Thus we conclude that $\Ucz$ is free over $\cz$. The second claim of the lemma is now immediate.
\end{proof}

\begin{lemma} \label{lemma:noetherian}
Let $B$ be a noetherian $\cz$-algebra. Then $U_B$ is noetherian.
\end{lemma}

\begin{proof}[Proof sketch]
The argument is due to Brown and Goodearl \cite[\S 2.2]{Brown:1997}. In \cite[\S 10.1]{De-Concini:1993}, De~Concini and Procesi define a sequence of degenerations
\begin{equation} \label{eq:degenerations}
\Uq = U^{(0)},U^{(1)},\ldots,U^{(2N)}
\end{equation}
of the algebra $\Uq$, each of which is the associated graded ring of the previous algebra with respect to a multiplicative $\N$-filtration. The definition of the degenerations relies on the commutation relations between the root vectors in $\Uq$. Since the root vectors in $\Uq$ are elements of $\Ucz$ by our choice for the denominator set $S$, one can define a similar sequence of degenerations
\[
U_B = U_B^{(0)},U_B^{(1)},\ldots,U_B^{(2N)}
\]
of the algebra $U_B$ such that $U_B^{(2N)}$ is an iterated twisted polynomial ring over the torus $U_B^0$. The torus $U_B^0$ is generated as a $B$-algebra by the finite set of commuting elements $\{K_i^{\pm 1},[K_i;0]: i \in [1,r] \}$, where $[K_i;0]:=E_iF_i-F_iE_i$ (cf.\ \cite[(1.5.4)]{De-Concini:1990}), so is noetherian because $B$ is noetherian. Then $U_B$ is noetherian by \cite[Theorems 1.2.9 and 1.6.9]{McConnell:1987}.
\end{proof}

\begin{corollary} \label{corollary:freeresolution}
Let $B$ be a $\cz$-algebra. There exists a resolution of the trivial $U_B$-module $B$ by finitely-generated free $U_B$-modules:
\[
\cdots \rightarrow P_n \rightarrow \cdots \rightarrow P_1 \rightarrow P_0 \rightarrow B \rightarrow 0.
\]
\end{corollary}

\begin{proof}
First consider the case $B = \cz$. Set $P_{-1} = \cz$, $P_0 = \Ucz$, and let $P_0 \rightarrow P_{-1}$ be the augmentation map. Now given $P_n$ with $n \geq 0$, let $I_n$ be the kernel of the map $P_n \rightarrow P_{n-1}$. Since by induction $P_n$ is a finitely-generated $\Ucz$-module, and since $\Ucz$ is noetherian by Lemma \ref{lemma:noetherian}, the $\Ucz$-submodule $I_n$ of $P_n$ is also finitely-generated as a $\Ucz$-module. Then there exists a finitely-generated free $\Ucz$-module $P_{n+1}$ mapping onto $I_n$. Take $P_{n+1} \rightarrow P_n$ to be the composite map $P_{n+1} \twoheadrightarrow I_n \hookrightarrow P_n$. We thus inductively construct the resolution $P_\bullet \rightarrow \cz$ of $\cz$ by finitely-generated free $\Ucz$-modules. Since $\Ucz$ is free over $\cz$ by Lemma \ref{lemma:free}, $P_\bullet \rightarrow \cz$ is a complex of free $\cz$-modules, hence splits over $\cz$. It then follows for any $\cz$-algebra $B$ that $P_\bullet \ocz B \rightarrow B$ is a resolution of $B$ by finitely-generated free $U_B$-modules.
\end{proof}

\subsection{Base change and the universal coefficient theorem} \label{subsection:basechange}

The crux of our argument for computing the cohomology ring $\Hbul(\Uq,k)$ relies on the universal coefficient theorem, which we now recall.

\begin{theorem}[Universal Coefficient Theorem for Homology] \textup{\cite[Theorem 7.55]{Rotman:2009}} \label{theorem:UCT}
Let $R$ be a ring, $A$ a left $R$-module, and $(K,d)$ a chain complex of flat right $R$-modules such that the subcomplex of boundaries also consists of flat $R$-modules. Then for each $n \in \Z$, there exists a short exact sequence
\begin{equation} \label{eq:UCT}
0 \rightarrow \opH_n(K) \otimes_R A \stackrel{\lambda_n}{\rightarrow} \opH_n(K \otimes_R A) \stackrel{\mu_n}{\rightarrow} \Tor_1^R(\opH_{n-1}(K),A) \rightarrow 0,
\end{equation}
natural with respect to both $K$ and $A$, such that $\lambda_n: \cls(z) \otimes a \mapsto \cls(z \otimes a)$.
\end{theorem}

We apply the universal coefficient theorem as follows:

\begin{lemma} \label{lemma:ses}
Let $B$ be a $\cz$-algebra, and $\Gamma$ a $B$-algebra. Suppose $B$ is a principal ideal domain. Then for each $n \in \N$, there exists a short exact sequence
\begin{equation}  \label{eq:ses}
0 \rightarrow \opH^n(U_B,B) \otimes_B \Gamma \stackrel{\lambda_n}{\rightarrow} \opH^n(U_\Gamma,\Gamma) \stackrel{\mu_n}{\rightarrow} \Tor_1^B(\opH^{n+1}(U_B,B),\Gamma) \rightarrow 0,
\end{equation}
and the induced map $\lambda: \Hbul(U_B,B) \otimes_B \Gamma \rightarrow \Hbul(U_\Gamma,\Gamma)$ is an algebra homomorphism.
\end{lemma}

\begin{proof}
Let $P_\bullet \rightarrow B$ be a resolution of $B$ by finitely-generated free $U_B$-modules as in Corollary \ref{corollary:freeresolution}, and set $K_n = \Hom_{U_B}(P_{-n},B)$. Then the chain complex $K_\bullet$ consists of finitely-generated free $B$-modules. Since every submodule of a free module over a PID is again free, the subcomplex of boundaries in $K$ is also free, hence flat, over $B$. Also, since $P_n$ is free over $U_B$, there exists for each $n \in \N$ a natural isomorphism
\begin{equation} \label{eq:extendtogamma}
\Hom_{U_B}(P_n,B) \otimes_B \Gamma \cong \Hom_{U_\Gamma}(P_n \otimes_B \Gamma,\Gamma).
\end{equation}
Then applying the universal coefficient theorem with $R = B$, $A = \Gamma$, and $K$ as above, one obtains the short exact sequence \eqref{eq:ses}.

Now let $\alpha \in \opH^a(U_B,B)$ and $\beta \in \opH^b(U_B,B)$ be represented by cocycles $f_\alpha \in K_{-a}$ and $f_\beta \in K_{-b}$, respectively, and let $\Delta: P \rightarrow P \otimes_B P$ be a $U_B$-module chain map lifting the isomorphism $B \cong B \otimes_B B$. Then the product $\alpha \beta$ is represented by the cocycle $(f_\alpha \otimes_B f_\beta) \circ \Delta \in K_{-(a+b)}$. Observe that $\Delta \otimes \id _\Gamma: P \otimes_B \Gamma \rightarrow (P \otimes_B P) \otimes_B \Gamma \cong (P \otimes_B \Gamma) \otimes_\Gamma (P \otimes_B \Gamma)$ is a chain map lifting the isomorphism $\Gamma \cong \Gamma \otimes_\Gamma \Gamma$. Then making the identification \eqref{eq:extendtogamma}, one sees for all $\gamma_\alpha,\gamma_\beta \in \Gamma$ that $\lambda(\alpha\beta \otimes_B \gamma_\alpha \gamma_\beta)$ and the product $\lambda_a(\alpha \otimes \gamma_\alpha)\lambda_b(\beta \otimes \gamma_\beta)$ are both represented by the cocycle
\[
[(f_\alpha \otimes_B f_\beta) \circ \Delta] \otimes_B \gamma_\alpha \gamma_\beta \in \Hom_{U_B}(P_{a+b},B) \otimes_B \Gamma,
\]
and hence that $\lambda$ is an algebra homomorphism.
\end{proof}

In Lemma \ref{lemma:ses} we assumed that $B$ was a principal ideal domain to conclude that the subcomplex of boundaries in $K$ was flat. This conclusion would also hold under the weaker assumption that $B$ is right semihereditary, or perhaps under even weaker assumptions on $B$, but we will not require such a generalization in this paper.

We now collect some results useful for analyzing the $\Tor$-group in \eqref{eq:ses}.

\begin{lemma} \label{lemma:finitelygenerated}
Let $B$ be a noetherian $\cz$-algebra. Then for each $n \in \N$, the cohomology group $\opH^n(U_B,B)$ is a finitely-generated $B$-module.
\end{lemma}

\begin{proof}
Let $K = \Hom_{U_B}(P_\bullet,B)$ be the complex of finitely-generated free $B$-modules considered in the proof of Lemma \ref{lemma:ses}. Since $B$ is noetherian, any subquotient of a finitely-generated $B$-module is again finitely-generated. In particular, $\opH^n(U_B,B)$ is a $B$-module subquotient of $K_{-n}$, so is finitely-generated over $B$.
\end{proof}

\begin{lemma} \label{lemma:tor}
Let $B$ be a commutative noetherian local ring with maximal ideal $\mathfrak{m}$, and let $M$ be a finitely-generated $B$-module. Then $M$ is a free $B$-module if and only if $\Tor_1^B(M,B/\mathfrak{m}) = 0$.
\end{lemma}

\begin{proof}
This follows from \cite[II.3.2 Corollary 2 of Proposition 5]{Bourbaki:1998}.
\end{proof}

In a similar vein, one has:

\begin{lemma} \label{lemma:notor}
Let $B$ be an integral domain, $b \in B$, and $M$ a $B$-module. Then
\[
\Tor_1^B(M,B/bB) \cong \set{m \in M: b.m = 0}.
\]
\end{lemma}

\begin{proof}
Compute the $\Tor$-group using the resolution $0 \rightarrow B \stackrel{\times b}{\rightarrow} B \rightarrow B/bB \rightarrow 0$.
\end{proof}

\subsection{The integral form \texorpdfstring{$\Ua$}{UA}} \label{subsection:integralform}

We now define the integral form $\Ua$, and describe how we will apply the results of Section \ref{subsection:basechange} to relate the cohomology theories for $\Uq$, $\Ua$, and $\Ug$. To begin, set $\A = \C[q]_{(q-1)}$, the localization of $\C[q]$ at the maximal ideal generated by $q-1$. Then $\A$ is a local principal ideal domain, with quotient field $k = \Cq$ and residue field $\C$. As in Section \ref{subsection:notation}, we write $\C_1$ for the field $\C$ considered as an $\A$-algebra via the map $q \mapsto 1$.

The field $k$ is $\A$-flat by \cite[Corollary 3.50]{Rotman:2009} because it is torsion-free, so applying Lemma \ref{lemma:ses} with $B = \A$ and $\Gamma = k$, we get for each $n \in \N$ the isomorphism
\begin{equation} \label{eq:extendscalarsiso}
\opH^n(\Ua,\A) \oa k \cong \opH^n(\Uq,k).
\end{equation}
On the other hand, $U_1 = \Ua \oa \C_1$, so applying Lemma \ref{lemma:ses} with $B = \A$ and $\Gamma = \C_1$, we get for each $n \in \N$ the short exact sequence
\begin{equation} \label{eq:sesA}
0 \rightarrow \opH^n(\Ua,\A) \oa \C_1 \stackrel{\lambda_n}{\rightarrow} \opH^n(U_1,\C) \rightarrow \Tor_1^\A(\opH^{n+1}(\Ua,\A),\C_1) \rightarrow 0.
\end{equation}
It follows from Lemmas \ref{lemma:finitelygenerated} and \ref{lemma:tor} that the map $\lambda_n$ is an isomorphism if and only if $\opH^{n+1}(\Ua,\A)$ is free as an $\A$-module. In particular, if the algebra homomorphism $\lambda: \Hbul(\Ua,\A) \oa \C_1 \rightarrow \Hbul(U_1,\C)$ is an isomorphism, then for each $n \in \N$, $\opH^n(\Ua,\A)$ must be $\A$-free of rank $\dim_\C \opH^n(U_1,\C) = \dim_\C \opH^n(\Ug,\C)$.

Our strategy for computing $\Hbul(\Uq,k)$ is now as follows. We first verify that the injective algebra homomorphism $\lambda: \Hbul(\Ua,\A) \oa \C_1 \rightarrow \Hbul(U_1,\C)$ is an isomorphism, and hence that $\Hbul(\Ua,\A)$ is $\A$-free of rank $\dim_\C \Hbul(\Ug,\C)$, by showing that the odd degree homogeneous generators for $\Hbul(U_1,\C) \cong \Hbul(\Ug,\C)$ all lie in the image of $\lambda$. We verify this for $\g$ not of type $D_r$ or $E_6$ in Section \ref{subsection:Uafree}, and for types $D_r$ and $E_6$ in Sections \ref{subsection:typeD} and \ref{subsection:typeE6}. Next, using the fact that $\Hbul(\Ua,\A)$ is $\A$-free and that $\Hbul(\Ua,\A) \oa \C_1 \cong \Hbul(\Ug,\C)$ is an exterior algebra, we deduce in Section \ref{subsection:ringstructure} that $\Hbul(\Ua,\A)$ is an exterior algebra generated in the same odd degrees as is $\Hbul(\Ug,\C)$. Finally, we apply \eqref{eq:extendscalarsiso} to deduce the structure of $\Hbul(\Uq,k)$.

\subsection{Cohomology for \texorpdfstring{$\Ua$}{UA}} \label{subsection:Uafree}

Following the strategy outlined in Section \ref{subsection:integralform}, we first verify that $\lambda$ is an isomorphism when $\g$ is not of type $D_r$ or $E_6$.

\begin{theorem} \label{theorem:free}
Suppose $\g$ is not of type $D_r$ or $E_6$. Then the injective algebra map
\[
\lambda: \Hbul(\Ua,\A) \oa \C_1 \rightarrow \Hbul(U_1,\C)
\]
is an isomorphism. In particular, $\Hbul(\Ua,\A)$ is a finitely-generated free $\A$-module.
\end{theorem}

\begin{proof}
We prove the theorem by showing that the odd-degree homogeneous generators for $\Hbul(U_1,\C)$ described in Theorem \ref{theorem:exterioralgebra} all lie in the image of $\lambda$. First suppose $\g$ is of type $A_1$, $A_2$, $B_2$, $C_2$, $E_7$, $E_8$, $F_4$, or $G_2$, and let $n$ be one of the odd degrees listed in Table \ref{table:degrees}. Using Theorem \ref{theorem:exterioralgebra} and Table \ref{table:degrees} one can check that $\opH^{n+1}(\Ug,\C) = 0$. Then \eqref{eq:sesA} implies that
\[
\opH^{n+1}(\Ua,\A)/(q-1)\opH^{n+1}(\Ua,\A) \cong \opH^{n+1}(\Ua,\A) \oa \C_1 = 0
\]
and hence $\opH^{n+1}(\Ua,\A) = 0$ by Nakayama's Lemma. Then $\lambda_n: \opH^n(\Ua,\A) \oa \C_1 \rightarrow \opH^n(U_1,\C)$ is an isomorphism by \eqref{eq:sesA}, so for these Lie types we conclude that the odd-degree homogeneous generators for $\Hbul(U_1,\C)$ all lie in the image of $\lambda$.

Now suppose that $\g$ is of type $X_r$, with $X \in \set{A,B,C}$ and $r \geq 3$. Let $\g' \subset \g$ be the subalgebra of $\g$ of type $X_{r-1}$ as defined in cases (1) and (2) of Theorem \ref{theorem:restrictionmap}. Define $\Uq(\g')$ and $\Ua(\g')$ to be the subalgebras of $\Uq$ and $\Ua$, respectively, generated by the set $\{E_i,F_i,K_i^{\pm 1}:i \in [2,r] \}$. Then $\Uq(\g')$ is isomorphic to the quantized enveloping algebra associated to $\g'$, and $\Ua(\g')$ is its corresponding integral form. By induction on the rank of $\g$, we may assume for each $n \in \N$ that the space $\opH^n(\Ua(\g'),\A)$ is $\A$-free of rank $\dim_\C \opH^n(U(\g'),\C)$. Let $n_1 < \cdots < n_r$ be the degrees listed in Table \ref{table:degrees} of the homogeneous generators for $\Hbul(\Ug,\C) \cong \Hbul(U_1,\C)$. As in Theorem \ref{theorem:restrictionmap}, write $\Hbul(\Ug,\C) \cong \Lambda(x_{n_1},\ldots,x_{n_r})$, with $x_{n_i}$ of degree $n_i$, and set $z_i = x_{n_i}$. Let $j \in [1,r]$, and assume by induction that $z_1,\ldots,z_{j-1} \in \im(\lambda)$. To show that $z_j \in \im(\lambda)$, it suffices to show that $\opH^{n_j+1}(\Ua,\A)$ is $\A$-free, since this implies by \eqref{eq:sesA} that $\lambda_{n_j}:\opH^{n_j}(\Ua,\A) \oa \C_1 \rightarrow \opH^{n_j}(U_1,\C)$ is an isomorphism.

By Theorem \ref{theorem:exterioralgebra}, the space $\opH^{n_j+1}(U_1,\C)$ is spanned by certain monomials in the generators $z_1,\ldots,z_r$, but since $n_i \neq 1$ for any $i$, no nonzero monomial can involve a generator $z_i$ with $i \geq j$. Then $\opH^{n_j+1}(U_1,\C)$ is spanned by certain monomials in the generators $z_1,\ldots,z_{j-1} \in \im(\lambda)$, and it follows that these monomials are in the image of $\lambda$, and hence that $\lambda_{n_j+1}$ is an isomorphism. Now consider the following diagram, where the vertical arrows are the corresponding restriction maps:
\begin{equation} \label{eq:diagram}
\xymatrix{
\opH^{n_j+1}(\Ua(\g),\A) \oa \C_1 \ar@{->}[r]^(.55){\lambda_{n_j+1}}_(.55){\sim} \ar@{->}[d] & \opH^{n_j+1}(\Ug,\C) \ar@{->}[d] \\
\opH^{n_j+1}(\Ua(\g'),\A) \oa \C_1 \ar@{->}[r]^(.55){\lambda_{n_j+1}}_(.55){\sim} & \opH^{n_j+1}(U(\g'),\C) 
}
\end{equation}
The commutativity of the diagram follows from the fact that the universal coefficient theorem (Theorem \ref{theorem:UCT}) is natural with respect to the complex $K$. The bottom map in the diagram is an isomorphism by induction on the rank of the Lie algebra. The right-hand restriction map is also an isomorphism, since by Theorem \ref{theorem:restrictionmap} the homogeneous generators $z_1,\ldots,z_{j-1}$ for $\Hbul(\Ug,\C)$ can be chosen so that the restriction map $\Hbul(\Ug,\C) \rightarrow \Hbul(U(\g'),\C)$ maps them onto the corresponding generators for $\Hbul(U(\g'),\C)$. This implies that the left-hand restriction map is an isomorphism as well, hence that the map $\opH^{n_j+1}(\Ua(\g),\A) \rightarrow \opH^{n_j+1}(\Ua(\g'),\A)/(q-1) \opH^{n_j+1}(\Ua(\g'),\A)$ is surjective. Then the restriction map $\opH^{n_j+1}(\Ua,\A) \rightarrow \opH^{n_j+1}(\Ua(\g'),\A)$ is surjective by Nakayama's Lemma. By induction on the rank of the Lie algebra, the space $\opH^{n_j+1}(\Ua(\g'),\A)$ is $\A$-free of rank $\dim_\C \opH^{n_j+1}(U(\g'),\C) = \dim_\C \opH^{n_j+1}(\Ug,\C)$. Then the map $\opH^{n_j+1}(\Ua,\A) \rightarrow \opH^{n_j+1}(\Ua(\g'),\A)$ is a split surjection of $\A$-modules, and $\opH^{n_j+1}(\Ua,\A)$ has $\A$-rank at least $\dim_\C \opH^{n_j+1}(\Ug,\C)$. Now
\begin{align*}
\dim_\C \opH^{n_j+1}(\Ug,\C) &\leq \dim_k \opH^{n_j+1}(\Ua,\A) \oa k & \text{by the bound on the $\A$-rank,} \\
&\leq \dim_\C \opH^{n_j+1}(\Ua,\A) \oa \C_1 \\
&= \dim_\C \opH^{n_j+1}(\Ug,\C),
\end{align*}
so we conclude that $\opH^{n_j+1}(\Ua,\A)$ is $\A$-free by \cite[Lemma 1.21]{Andersen:1991}.
\end{proof}

\section{Cohomology for the quantized enveloping algebra \texorpdfstring{$\Uq$}{Uq}}

\subsection{Cohomology ring structure} \label{subsection:ringstructure}

We now deduce the structure of $\Hbul(\Uq,k)$ in any case for which $\lambda: \Hbul(\Ua,\A) \oa \C_1 \rightarrow \Hbul(U_1,\C)$ is an isomorphism.

\begin{theorem} \label{theorem:Uqcohomology}
Suppose that $\lambda: \Hbul(\Ua,\A) \oa \C_1 \rightarrow \Hbul(U_1,\C)$ is an isomorphism. Then the cohomology rings $\Hbul(\Ua,\A)$ and $\Hbul(\Uqg,k)$ are exterior algebras generated by homogeneous elements in the odd degrees listed in Table \ref{table:degrees}.
\end{theorem}

\begin{proof}
Since $\lambda$ is an isomorphism, we have for each $n \in \N$ that $\opH^n(\Ua,\A)$ is a free $\A$-module of rank $\dim_\C \opH^n(\Ug,\C)$ by the discussion in Section \ref{subsection:integralform}. Choose homogeneous elements $z_1,\ldots,z_r \in \Hbul(\Ua,\A)$ such that their images under $\lambda$ in $\Hbul(U_1,\C) \cong \Hbul(\Ug,\C)$ are the homogeneous generators described in Theorem \ref{theorem:exterioralgebra}. Since $\Ua$ is a Hopf algebra over the commutative ring $\A$, the cohomology ring $\Hbul(\Ua,\A)$ is graded-commutative \cite[Corollary VIII.4.3]{Mac-Lane:1995}. The elements $z_1,\ldots,z_r \in \Hbul(\Ua,\A)$ are each homogeneous of odd degree, so $z_i^2 = 0$ for each $i \in [1,r]$, and there exists a well-defined map $\varphi: \Lambda(z_1,\ldots,z_r) \rightarrow \Hbul(\Ua,\A)$ of graded $\A$-algebras. The induced map $\varphi \oa \C_1:\Lambda(z_1,\ldots,z_r) \oa \C_1 \rightarrow \Hbul(\Ua,\A) \oa \C_1$ is surjective by the choice of the $z_i$, so we conclude by Nakayama's Lemma that $\varphi$ is surjective, hence a graded algebra isomorphism because $\Lambda(z_1,\ldots,z_r)$ and $\Hbul(\Ua,\A)$ are each $\A$-free of the same finite rank. Extending scalars to $k$, we obtain via \eqref{eq:extendscalarsiso} the graded algebra isomorphism $\varphi \oa k: \Lambda(z_1,\ldots,z_r) \oa k \stackrel{\sim}{\rightarrow} \Hbul(\Uqg,k)$.
\end{proof}

\subsection{Type \texorpdfstring{$\mathbf{D}$}{D}} \label{subsection:typeD}

To extend Theorem \ref{theorem:free} to the case when $\g$ is of type $D_r$, we consider cohomological restrictions maps corresponding not only to a Lie subalgebra $\g'$ of $\g$ of type $D_{r-1}$, but also to a Lie subalgebra $\g''$ of $\g$ of type $A_{r-1}$. In the latter case, we also require the explicit understanding of the ring structure for $\Hbul(\Uq(\g''),k)$ that comes from Theorem \ref{theorem:Uqcohomology}.

\begin{theorem} \label{theorem:typeD}
The conclusion of Theorem \ref{theorem:free} holds if $\g$ is of type $D_r$.
\end{theorem}

\begin{proof}
Suppose $\g$ is of type $D_r$ with $r \geq 4$. The overall strategy is similar to that in the proof of Theorem \ref{theorem:free} for types $A$, $B$, and $C$, though some subtleties arise because the right-hand column of \eqref{eq:diagram} need not be an isomorphism when $\g$ is of type $D$. As in the proof of Theorem \ref{theorem:free}, we consider a subalgebra $\g' \subset \g$ of type $D_{r-1}$, as defined in case (3) of Theorem \ref{theorem:restrictionmap}, and also a subalgebra $\g'' \subset \g$ of type $A_{r-1}$, as defined in case (4) of Theorem \ref{theorem:restrictionmap}. (If $r = 4$, then $\g'$ is of type $A_3$, and cases (3) and (4) of Theorem \ref{theorem:restrictionmap} coincide.) For $j \in [1,r-1]$ set $n_j = 4j-1$, and set $n_r = 2r-1$, so that $n_1,\ldots,n_r$ are the degrees listed in Table \ref{table:degrees} for type $D_r$.

Our first step is to show for all $n \in [1,2r]$ that $\opH^n(\Ua,\A) \oa \C_1 \cong \opH^n(U_1,\C)$. Since $\Hbul(U_1,\C)$ is an exterior algebra generated in the odd degrees $n_1,\ldots,n_r$, this is equivalent to showing $\opH^{n_j}(\Ua,\A) \oa \C_1 \cong \opH^{n_j}(U_1,\C)$ whenever $n_j \leq 2r-1$. First let $j \in [1,r]$ with $n_j \leq 2r-3$. It follows from Theorem \ref{theorem:restrictionmap} that the restriction map $\opH^{n_j+1}(\Ug,\C) \rightarrow \opH^{n_j+1}(U(\g'),\C)$ is an isomorphism; cf.\ the analysis of \eqref{eq:diagram}. Also, by induction on the rank of $\g$, we may assume for all $n \in [1,2(r-1)]$ that $\opH^n(\Ua(\g'),\A) \oa \C_1 \cong \opH^n(U(\g'),\C)$, and hence that $\opH^n(\Ua(\g'),\A)$ is $\A$-free of rank $\dim_\C \opH^n(U(\g'),\C)$; cf.\ Section \ref{subsection:integralform}. Now one can imitate the proof of Theorem \ref{theorem:free}, arguing by induction on the rank and the degree, to show for all $n_j \leq 2r-3$ that $\opH^{n_j}(\Ua,\A) \oa \C_1 \cong \opH^{n_j}(U_1,\C)$. Then to complete the first step, we must now show that $\opH^{2r-1}(\Ua,\A) \oa \C_1 \cong \opH^{2r-1}(U_1,\C)$.

Given $y \in \Hbul(\Ua,\A)$, set $\ybar = \lambda(y \oa 1) \in \Hbul(U_1,\C)$. By the previous paragraph, we can choose $y_1,\ldots,y_s \in \Hbul(\Ua,\A)$ such that $\ybar_1,\ldots,\ybar_s \in \Hbul(U_1,\C)$ are representatives for the homogeneous generators for $\Hbul(U_1,\C)$ of degrees less than or equal to $2r-3$. Then $\opH^{2r}(U_1,\C)$ is spanned over $\C$ by certain monomials in the vectors $\ybar_1,\ldots,\ybar_s$. Let $m_1,\ldots,m_t \in \opH^{2r}(\Ua,\A)$ be monomials in the $y_i$ such that $\mbar_1,\ldots,\mbar_t$ form a basis for $\opH^{2r}(U_1,\C)$. We want to show that $\dim_k \opH^{2r}(\Uq,k) \geq t$, for this implies by \eqref{eq:extendscalarsiso} and \cite[Lemma 1.21]{Andersen:1991} that $\opH^{2r}(\Ua,A)$ is $\A$-free, and hence that $\opH^{2r-1}(\Ua,\A) \oa \C_1 \cong \opH^{2r-1}(U_1,\C)$ by \eqref{eq:ses}.

Let $\rho: \Hbul(\Uq,k) \rightarrow \Hbul(\Uq(\g''),k)$ be the restriction map. Given $y \in \Hbul(\Ua,\A)$, let $\wt{y}$ denote its image in $\Hbul(\Ua,\A) \oa k \cong \Hbul(\Uq,k)$. By Theorems \ref{theorem:free} and \ref{theorem:Uqcohomology}, $\Hbul(\Uq(\g''),k)$ is an exterior algebra generated by homogeneous elements of certain odd degrees. Moreover, it follows from Theorem \ref{theorem:restrictionmap} and the proof of Theorem \ref{theorem:Uqcohomology} that we can take certain of the generators for $\Hbul(\Uq(\g''),k)$ to be the vectors $\rho(\wty_1),\ldots,\rho(\wty_s)$. This implies that the vectors $\rho(\wtm_1),\ldots,\rho(\wtm_t) \in \opH^{2r}(\Uq(\g''),k)$ are linearly-independent, and hence $\wtm_1,\ldots,\wtm_t \in \opH^{2r}(\Uq,k)$ are as well. We then conclude that $\dim_k \opH^{2r}(\Uq,k) \geq t$, which completes the first step of the proof.

We have shown for all $a \in \N$ that if $\g$ is of type $D_a$, then $\opH^n(\Ua,\A) \oa \C_1 \cong \opH^n(U_1,\C)$ for $n \in [1,2a]$. Write $\Hbul(U_1,\C) \cong \Hbul(\Ug,\C) \cong \Lambda(x_3,\ldots,x_{4r-5},\wtx_{2r-1})$ as in Theorem \ref{theorem:restrictionmap}. Suppose $n_i = \deg(x_i) > 2r-1$; we must show that $x_i \in \im(\lambda)$. Set $m = 2r-2$, and let $\g_m$ be the finite-dimensional simple complex Lie algebra of type $D_m$. The inclusion of Dynkin diagrams $D_r \hookrightarrow D_m$ induces an inclusion of algebras $\Ua \hookrightarrow \Ua(\g_m)$; cf.\ Section \ref{subsection:Uafree}. We thus have the following commutative diagram, where the vertical arrows are the corresponding restriction maps:
\begin{equation} \label{eq:diagramtypeD}
\xymatrix{
\opH^{n_i}(\Ua(\g_m),\A) \oa \C_1 \ar@{->}[r]^(.57){\lambda_{n_i}} \ar@{->}[d] & \opH^{n_i}(U(\g_m),\C) \ar@{->}[d] \\
\opH^{n_i}(\Ua(\g),\A) \oa \C_1 \ar@{->}[r]^(.57){\lambda_{n_i}} & \opH^{n_i}(U(\g),\C) 
}
\end{equation}
Since $n_i \leq 4r-5 < 2m$, the top row of \eqref{eq:diagramtypeD} is an isomorphism by the first step of the proof. Also, since $n_i \leq 4r-5 \leq 4m-9$, it follows from Theorem \ref{theorem:restrictionmap} that $x_i$ is in the image of the restriction map $\Hbul(U(\g_m),\C) \rightarrow \Hbul(\Ug,\C)$. Then from the commutativity of \eqref{eq:diagramtypeD} we conclude that $x_i \in \im(\lambda)$. This completes the proof.
\end{proof}

\begin{corollary} \label{corollary:typeD}
The conclusion of Theorem \ref{theorem:Uqcohomology} holds if $\g$ is of type $D_r$.
\end{corollary}

\subsection{Type \texorpdfstring{$E_6$}{E6}} \label{subsection:typeE6}

To extend Theorem \ref{theorem:free} to the case when $\g$ is of type $E_6$, we consider restriction maps like those in cases (5) and (6) of Theorem \ref{theorem:restrictionmap}.

\begin{theorem} \label{theorem:typeE6}
Suppose $\g$ is of type $E_6$. Then $\Hbul(\Ua,\A) \oa \C_1 \cong \Hbul(U_1,\C)$.
\end{theorem}

\begin{proof}[Proof sketch]
The strategy is similar to the proofs of Theorems \ref{theorem:free} and \ref{theorem:typeD}. The generators for $\Hbul(U_1,\C)$ are in degrees $3$, $9$, $11$, $15$, $17$, and $23$. One can check using Theorem \ref{theorem:exterioralgebra} that $\opH^n(U_1,\C) = 0$ for $n \in \set{4,10,16}$, so $\opH^n(\Ua,\A) \oa \C_1 \cong \opH^n(U_1,\C)$ if $n \in \set{3,9,15}$. The description of the restriction map from $E_7$ to $E_6$ in case (7) of Theorem \ref{theorem:restrictionmap} implies that the generators of degrees $11$ and $23$ are also in $\im(\lambda)$; cf.\ the analysis of \eqref{eq:diagramtypeD}. Then it remains to show that $\opH^{17}(\Ua,\A) \oa \C_1 \cong \opH^{17}(U_1,\C)$, or equivalently that $\opH^{18}(\Ua,\A)$ is $\A$-free. We can choose $y_{3} \in \opH^3(\Ua,\A)$ and $y_{15} \in \opH^{15}(\Ua,\A)$ such that the product $\ybar_3 \ybar_{15}$ spans $\opH^{18}(U_1,\C)$. Let $\g' \subset \g$ be the subalgebra of type $D_5$ as defined in case (5) of Theorem \ref{theorem:restrictionmap}, and let $\rho: \Hbul(\Ua,\A) \rightarrow \Hbul(\Ua(\g'),\A)$ be the corresponding restriction map. Then the argument in the third paragraph of the proof of Theorem \ref{theorem:free} shows that $\rho$ is surjective in degrees $3$ and $15$. This implies by the proof of Theorem \ref{theorem:Uqcohomology} that $\opH^{18}(\Ua(\g'),\A) \oa k \cong \opH^{18}(\Uq(\g'),k) \cong k$ is spanned by $\rho(y_3 y_{15})$. Then the product $y_3 y_{15} \in \opH^{18}(\Ua,\A)$ must span a one-dimensional subspace of $\opH^{18}(\Uq,k)$. Now
\[
1 \leq \dim_k \opH^{18}(\Ua,\A) \oa k \leq \dim_\C \opH^{18}(\Ua,\A) \oa \C_1 \leq \dim_\C \opH^{18}(U_1,\C) = 1,
\]
so $\opH^{18}(\Ua,\A)$ must be $\A$-free or rank $1$ by \cite[Lemma 1.21]{Andersen:1991}.
\end{proof}

Here is the main result of our computations:

\begin{theorem}
The cohomology ring $\Hbul(\Uq,k)$ is an exterior algebra over a graded subspace with odd gradation. Explicitly, $\Hbul(\Uq,k)$ is generated as an exterior algebra by homogeneous elements in the same odd degrees as for $\Hbul(\Ug,\C)$.
\end{theorem}

\subsection{The third cohomology group}

A famous theorem of Chevalley and Eilenberg states that $\opH^3(\Ug,\C) \neq 0$ \cite[Theorem 21.1]{Chevalley:1948}. They prove the non-vanishing of $\opH^3(\Ug,\C)$ by showing that the Killing form on $\g$ gives rise to a nonvanishing invariant 3-cochain in $\g$. Our analysis gives us:

\begin{corollary}
Let $\g$ be a finite-dimensional simple complex Lie algebra. Then
\[
\dim_k \opH^3(\Uqg,k) = 1.
\]
\end{corollary}

It is an interesting question whether the non-vanishing of $\opH^3(\Uqg,k)$ could also be established in a manner similar to that of Chevalley and Eilenberg, perhaps by using the non-degenerate inner product on $\Uqg$ constructed by Rosso \cite{Rosso:1990a}.

\section{Cohomology for the specializations \texorpdfstring{$\Ue$}{Ue}}

\subsection{Generic behavior}

Recall the set $S$ and the ring $\ca = S^{-1} \C[q,q^{-1}]$ defined in Section \ref{subsection:notation}. We call $\ve \in \C$ a \emph{bad root of unity} if $\ve \in \set{\pm 1}$ or if $\ve$ is the root of some polynomial in $S$. Define the set $\Cg \subset \C$ by
\[
\Cg = \set{ \ve \in \C^\times: \text{$\ve$ is not a bad root of unity}}.
\]
Then for all $\ve \in \Cg$, the field $\C$ is an $\ca$-algebra via the map $q \mapsto \ve$, and we can apply the results of Section \ref{subsection:basechange} with $\B = \ca$ and $\Gamma = \Ce$.

\begin{proposition} \label{proposition:cafingen}
The ring $\Hbul(\Uca,\ca)$ is a finitely-generated $\ca$-module.
\end{proposition}

\begin{proof}
For each $n \in \N$, the space $\opH^n(\Uca,\ca)$ is a finitely-generated $\ca$-module by Lemma \ref{lemma:finitelygenerated}. Set $d = \dim_\C \g$. Then for all $\ve \in \Cg$, the ring $\Hbul(\Ue,\C)$ satisfies the Poincar\'{e} duality $\opH^n(\Ue,\C) \cong \opH_{d-n}(\Ue,\C)$ by \cite[Corollary 3.2.2]{Chemla:2004}. In particular, $\opH^n(\Ue,\C) = 0$ for all $n > d$.  This implies by Lemma \ref{lemma:ses} with $B = \ca$ and $\Gamma = \Ce$ that $\opH^n(\Uca,\ca) \oca \Ce = 0$ for all $n > d$. Since $\ve \in \Cg$ was arbitrary, it follows for $n > d$ (e.g., using the fundamental theorem for finitely-generated modules over a principal ideal domain) that $\opH^n(\Uca,\ca) = 0$. Then $\Hbul(\Uca,\ca) = \bigoplus_{n=1}^{n=d} \opH^n(\Uca,\ca)$, so $\Hbul(\Uca,\ca)$ is a finitely-generated $\ca$-module.
\end{proof}

\begin{corollary} \label{corollary:Uegeneric}
For all but finitely many $\ve \in \Cg$, $\Hbul(\Uca,\ca) \oca \Ce \cong \Hbul(\Ue,\C)$, and for all such $\ve \in \Cg$, $\Hbul(\Ue,\C)$ is generated as an exterior algebra by homogeneous elements in the same odd degrees as for $\Hbul(\Ug,\C)$.
\end{corollary}

\begin{proof}
It follows from Proposition \ref{proposition:cafingen} and the fundamental theorem for finitely-generated modules over a principal ideal domain that $\Hbul(\Uca,\ca)$ is $(q-\ve)$-torsion free for all but finitely many $\ve \in \Cg$. Let $S' \subset \ca$ be the multiplicatively closed set generated by
\[
\{(q-\ve) \in \C[q] : \Hbul(\Uca,\ca) \text{ has $(q-\ve)$-torsion}\},
\]
and set $\cb = (S')^{-1} \ca$. Since $\cb$ is flat over $\ca$, we have $\Hbul(\Uca,\ca) \oca \cb \cong \Hbul(\Ucb,\cb)$ by Lemma \ref{lemma:ses}, and we deduce that $\Hbul(\Ucb,\cb)$ is a free $\cb$-module (all of the torsion has been eliminated by the choice of denominator set). Since the ring $\A$ is a localization of $\cb$, we also have $\Hbul(\Ucb,\cb) \ocb \A \cong \Hbul(\Ua,\A)$ by Lemma \ref{lemma:ses}, and we can choose the odd-degree generators $z_1,\ldots,z_r$ for the exterior algebra $\Hbul(\Ua,\A)$ to be elements of $\Hbul(\Ucb,\cb)$. Then there exists a map of free $\cb$-algebras $\varphi: \Lambda(z_1,\ldots,z_r) \rightarrow \Hbul(\Ucb,\cb)$.

Set $W = \Hbul(\Ucb,\cb)/\im(\varphi)$. Since $\Lambda(x_1,\ldots,x_r)$ and $\Hbul(\Ucb,\cb)$ are each free of the same finite rank over the principal ideal domain $\cb$, $W$ is a finitely-generated torsion $\cb$-module. Let $\ve \in \Cg$ such that $(q-\ve) \notin S'$ and $W$ has no $(q-e)$-torsion. Then $\Tor_1^{\cb}(W,\Ce) = \Tor_1^{\cb}(W,\cb/(q-\ve)\cb) = 0$ by Lemma \ref{lemma:notor}, so it follows from the long exact sequence for $\Tor_1^{\cb}(-,\Ce)$ applied to the short exact sequence
\[
0 \rightarrow \Lambda(z_1,\ldots,z_r) \stackrel{\varphi}{\rightarrow} \Hbul(\Ucb,\cb) \rightarrow W \rightarrow 0
\]
that the algebra map $\varphi \ocb \Ce: \Lambda(x_1,\ldots,x_r) \ocb \Ce \rightarrow \Hbul(\Ucb,\cb) \ocb \Ce$ is injective. Then by dimension comparison $\varphi \ocb \Ce$ must also be surjective, hence an algebra isomorphism. Thus, the conclusion of the corollary holds for all $\ve \in \Cg$ such that $\Hbul(\Uca,\ca)$ and $W$ are $(q-\ve)$-torsion free, and fails for only finitely-many $\ve \in \Cg$.
\end{proof}

While Corollary \ref{corollary:Uegeneric} states for almost all values $\ve \in \Cg$ that $\Hbul(\Ue,\C)$ is an exterior algebra over an $r$-dimensional graded subspace, it unfortunately does not give any indication of the values for which this condition fails. We can at least say that the only values for which $\Hbul(\Ue,\C)$ might not be an exterior algebra are those $\ve$ that are algebraic over $\Q$. Indeed, let $B = S^{-1}\Q[q,q^{-1}]$, with $S$ as defined in Section \ref{subsection:notation}. Then for each $n \in \N$, the space $\opH^n(U_B,B)$ is a finitely-generated $B$-module by Lemma \ref{lemma:finitelygenerated}, and $\Hbul(U_B,B) \otimes_B \ca \cong \Hbul(\Uca,\ca)$. This shows that $\Hbul(\Uca,\ca)$ has $(q-\ve)$-torsion if and only if there exists an irreducible polynomial $f \in \Q[q]$ such that $(q-\ve)$ divides $f$ in $\C[q]$ and $\Hbul(U_B,B)$ has $f$-torsion. We summarize this discussion in the following proposition:

\begin{proposition}
If $\ve \in \Cg$ is transcendental over $\Q$, then $\Hbul(\Ue,\C)$ is an exterior algebra generated by homogeneous elements in the same odd degrees as for $\Hbul(\g,\C)$.
\end{proposition}

\subsection{Roots of unity} \label{subsection:rootsofunity}

Let $p$ be a prime, and let $h$ be the Coxeter number of the root system associated to $\g$. We can show that the conclusion of Corollary \ref{corollary:Uegeneric} holds for $\Ue$ provided $\ve$ is a primitive $p$-th root of unity and $p > 3(h-1)$.

\begin{theorem} \label{theorem:rootofunity}
Let $\ve \in \C$ be a primitive $p$-th root of unity with $p > 3(h-1)$. Then $\Hbul(\Ue,\C)$ is an exterior algebra generated by homogeneous elements in the same odd degrees as for $\Hbul(\Ug,\C)$.
\end{theorem}

\begin{proof}[Proof sketch]
The theorem is established by a sequence of arguments completely analogous to those used for the case when the parameter of the quantized enveloping algebra is an indeterminant, except that instead of relating $\Hbul(\Ue,\C)$ to Lie algebra cohomology in characteristic zero, we relate $\Hbul(\Ue,\C)$ to Lie algebra cohomology in characteristic $p$. Let $\Uq'$ be the algebra over $\Q(q)$ defined by the same generators and relations as for $\Uq$, and let $\Ue'$ be the algebra over $\Qe$ obtained by replacing $q$ in the definition of $\Uq'$ by $\ve$. Then $\Ue' \otimes_{\Qe} \C \cong \Ue$ and $\Hbul(\Ue',\Qe) \otimes_{\Qe} \C \cong \Hbul(\Ue,\C)$, so it suffices to show that $\Hbul(\Ue',\Qe)$ is an exterior algebra generated by homogeneous elements in the same odd degrees as for $\Hbul(\Ug,\C)$.

Let $\Fp$ be the field with $p$ elements, and consider the map $\pi: \Z[q] \rightarrow \Fp$ that takes $q \mapsto 1$. Let $\phi_p(q) = q^{p-1} + \cdots + q + 1$ be the $p$-th cyclotomic polynomial. Then $\phi_p(1) = p$, and $\pi$ factors through a map $\pi': \Z[\ve] \cong \Z[q]/(\phi_p) \rightarrow \Fp$. Let $\cz'$ be the localization of $\Z[\ve]$ at the maximal ideal $\ker \pi'$. The ring $\Z[\ve]$ is a noetherian Dedekind domain (because the ring of integers in an algebraic number field is always a Dedekind domain), hence so is the localization $\cz'$. A local Dedekind domain is a principal ideal domain, so we can apply the results of Sections \ref{subsection:resolution}--\ref{subsection:basechange} to the $\cz$-algebra $\cz'$, its quotient field $\Qe$, and its residue field $\Fp$.

Let $\gfp$ be the Lie algebra over $\Fp$ obtained by extension of scalars from a Chevalley basis for $\g$. Then $U_{\cz'} \otimes_{\cz'} \Fp$ is a central extension of the universal enveloping algebra $U(\gfp)$ by the group algebra over $\Fp$ for the finite group $G = (\Z/2\Z)^r$. Since $p$ is odd, the group algebra $\Fp G$ is a semisimple ring. Then as in Lemma \ref{lemma:U1toUgiso}, we get $\Hbul(U_{\cz'} \otimes_{\cz'} \Fp,\Fp) \cong \Hbul(U(\gfp),\Fp)$. Since $p > 3(h-1)$, the latter ring is an exterior algebra generated by homogeneous elements in the same odd degrees as for $\Hbul(\Ug,\C)$ \cite[Theorem 1.2]{Friedlander:1986a}. Now one argues as in Sections \ref{subsection:Uafree}--\ref{subsection:typeE6} to show that $\Hbul(U_{\cz'},\cz')$ is a finitely-generated free $\cz'$-module, and that $\Hbul(U_{\cz'},\cz')$ and $\Hbul(U_{\cz'},\cz') \otimes_{\cz'} \Qe \cong \Hbul(\Ue',\Qe)$ are exterior algebras over graded subspaces concentrated in the correct odd degrees.
\end{proof}

The lower bound of $3(h-1)$ in the above theorem is not sharp. The bound is made in order to guarantee that the cohomology ring $\Hbul(U(\gfp),\Fp)$ is an exterior algebra generated in the correct degrees. We have conducted computer calculations to compute the structure of $\Hbul(U(\gfp),\Fp)$ when $p$ is small and when $\g$ is of type $A_1$, $A_2$, $B_2$, or $G_2$, and have determined in these cases that it is sufficient to assume $p > h$. Though we suspect that $\Hbul(U(\gfp),\Fp)$ should be an exterior algebra provided only $p>h$, we have no proof of this claim at this time.

\subsection{Conjectures}

If $\g = \mathfrak{sl}_2(\C)$ and $\ve \in \Cg$, then it follows from Poincar\'{e} duality and \cite[Remark 7.17]{Masuoka:2008} that $\Hbul(\Ue,\C)$ is an exterior algebra generated by a vector in degree $3$, but for higher ranks it is not clear (at least, it is not clear to the author) how to proceed in general, even for specific values of $\ve \in \C^\times$. If $\ve \in \C^{\times}$ is not a root of unity, then it is well-known that the categories of finite-dimensional type-$1$ modules for $\Uq$ and $\Ue$ are both equivalent to the category of finite-dimensional $\g$-modules. The BGG categories for $\Uq$ and $\Ue$ are also both equivalent to the integral block of the BGG category $\mathcal{O}$ for $\g$ \cite[Remark 6.3]{Andersen:2011}. One might hope then to extend this equivalence to a larger subcategory of infinite-dimensional $\Ue$-modules containing the trivial module, and thereby prove the following conjecture:

\begin{conjecture} \label{conjecture:nonroot}
Suppose $\ve \in \Cq$ is not a root of unity. Then $\Hbul(\Ue,\C)$ is an exterior algebra generated in the same odd degrees as for $\Hbul(\Ug,\C)$.
\end{conjecture}

In establishing the fact that the inclusion map $\Lambda^\bullet(\g^*)^\g \rightarrow \Lambda^\bullet(\g^*)$ induces an isomorphism $\Lambda(\g^*)^\g \cong \Hbul(\Ug,\C)$, one uses the complete-reducibility of finite-dimensional $\g$-modules to conclude that $\Lambda^\bullet(\g^*)^\g$ is a $\g$-module summand in the space of cocycles in $\Lambda^\bullet(\g^*)$, and hence that $\Lambda^\bullet(\g^*)^\g \cong \Hbul(\Ug,\C)$. If one could explicitly construct a finite-dimensional complex $P$ computing $\Hbul(\Ue,\C)$ such that each term in $P$ was a $\Ue$-module (i.e., a quantum version of the Koszul complex), then one could try to imitate the classical approach, at least for $\ve$ not a root of unity, to try to understand the structure of the cohomology ring $\Hbul(\Ue,\C)$.

While no one has yet constructed a quantum analogue for the Koszul complex, it may be possible to find a suitable substitute by considering the sequence of May spectral sequences arising from the algebra degenerations \eqref{eq:degenerations} of De Concini and Procesi. Indeed, we have successfully used this approach in \cite[\S 5.4]{Drupieski:2011b} to help deduce the ring structure of the cohomology ring for the nilpotent subalgebra $\Ue^-$.

Finally, based on the results of Theorem \ref{theorem:rootofunity}, and on the comments made in the last paragraph of Section \ref{subsection:rootsofunity}, we offer the following conjecture for the structure of $\Hbul(\Ue,\C)$ when $\ve \in \C$ is a root of unity:

\begin{conjecture}
Let $\ell$ be an odd positive integer, with $\ell$ coprime to $3$ if $\g$ is of type $G_2$. Let $\ve \in \C$ be a primitive $\ell$-th root of unity, and suppose $\ell > h$. Then $\Hbul(\Ue,\C)$ is an exterior algebra generated in the same odd degrees as for $\Hbul(\Ug,\C)$.
\end{conjecture}

\section*{Ackowledgements}

The author expresses his gratitude to Brian Parshall for suggesting this problem, and to Nick Kuhn for useful conversations on the cohomology of Lie groups.


\providecommand{\bysame}{\leavevmode\hbox to3em{\hrulefill}\thinspace}

\end{document}